\documentclass[11pt]{article}
\usepackage[T1]{fontenc}
\usepackage[utf8]{inputenc}
\usepackage{amsmath,amssymb,amsthm,mathtools,mathrsfs}
\usepackage{geometry}
\usepackage{hyperref}
\usepackage{cleveref}
\usepackage[
backend=biber,
style=numeric,
giveninits=true,
url = false,
doi = false,
isbn =false,
maxbibnames=99,
maxalphanames=99,
]{biblatex}
\renewbibmacro{in:}{}
\DeclareFieldFormat{title}{#1} 
\DeclareFieldFormat[article]{title}{#1}
\DeclareFieldFormat[inbook]{title}{#1}
\DeclareFieldFormat[incollection]{title}{#1}

\AtBeginBibliography{\small}
\hypersetup{colorlinks=true, linkcolor=blue, citecolor=blue, urlcolor=blue}

\newtheorem{theorem}{Theorem}[section]
\newtheorem{proposition}[theorem]{Proposition}

\theoremstyle{definition}
\newtheorem{remark}[theorem]{Remark}

\DeclareMathOperator{\sech}{sech}
\DeclareMathOperator{\supp}{supp}

\newcommand{\E}{\mathbb E}
\newcommand{\Pp}{\mathbb P}
\newcommand{\R}{\mathbb R}

\newcommand{\Par}{\mathcal P}

\title{Replica symmetry up to the de Almeida--Thouless line in the Sherrington--Kirkpatrick model}
\author{Patrick Lopatto}
\date{\today}

\begin{document}
\maketitle

\begin{abstract}
We show that in the Sherrington--Kirkpatrick
model at inverse temperature $\beta$ with uniform external field $h>0$, replica symmetry holds for the set of pairs $(\beta, h)$ predicted by de Almeida and Thouless (1978). The proof proceeds by a direct analysis of the Parisi measure using the characterization provided by Jagannath and Tobasco (2017).
\end{abstract}


\section{Introduction}

The Sherrington--Kirkpatrick model has been highly influential in the study of spin glasses; see \cite{chatterjee2026michel} for a recent review. At inverse temperature $\beta>0$ and under a uniform external
field with strength $h$, it has Hamiltonian
\[
H_N(\sigma)=\frac{\beta}{\sqrt N}\sum_{1\le i<j\le N}g_{ij}\sigma_i\sigma_j
+h\sum_{i=1}^N \sigma_i,
\qquad \sigma\in\{-1,1\}^N,
\]
where the couplings $(g_{ij})_{i<j}$ are independent, identically distributed standard Gaussians.  Let
\[
F_N(\beta,h)=\frac1N\E\left[ \log\sum_{\sigma\in\{-1,1\}^N}e^{H_N(\sigma)} \right],
\]
and let $F(\beta,h)$ denote the large-$N$  limit of $F_N(\beta,h)$, whose existence is
recalled in Section~\ref{s:preliminaries}.

We are interested in understanding for what values of $\beta$ and $h$ this model undergoes replica symmetry breaking. 
The $h=0$ case is now classical: in zero external field, the Sherrington--Kirkpatrick model is replica symmetric for $\beta\le 1$, while replica symmetry breaking occurs for $\beta>1$ \cite{aizenman1987some, toninelli2002almeida}. 
We therefore focus here on the case $h>0$. 

In the positive-field regime, there is a prediction due to de Almeida and Thouless \cite{de1978stability}. 
We let $q = q(\beta , h)  \in[0,1]$ denote the solution to
\begin{equation}\label{e:q-fp}
q=\E\big[\tanh^2(\beta\sqrt q Z+h) \big], \qquad Z\sim  N(0,1),
\end{equation}
recalling that the existence and uniqueness of $q$ for all $h>0$ are well known; see \cite[Proposition 1.3.8]{TalagrandBook1}. We define the parameter
\begin{equation}\label{e:alpha-def}
\alpha(\beta,h)=\beta^2\E\big[\sech^4(\beta\sqrt q Z+h)\big].
\end{equation}
The replica-symmetric ansatz for the limiting free energy $F(\beta, h)$ is 
\begin{equation}\label{e:PhiRS}
\Phi_{\mathrm{RS}}(\beta,h)
=\log 2+\E\big[ \log\cosh(\beta\sqrt q Z+h)\big]+\frac{\beta^2}{4}(1-q)^2.
\end{equation}
The de Almeida--Thouless prediction is that the replica-symmetric formula should hold throughout the region
$\alpha(\beta,h)\le 1$, with replica symmetry breaking on its complement. 

Previous efforts have yielded significant progress toward establishing this prediction. 
The work \cite{toninelli2002almeida}, mentioned previously in relation to the $h=0$ case, in fact showed that replica symmetry breaking holds whenever $\alpha ( \beta, h) >1$. 
In \cite{JT2017}, the authors showed that replica symmetry holds throughout the region $\alpha(\beta, h) < 1$, except possibly for a bounded exceptional set not containing the zero-field critical point $(\beta, h) = (1,0)$. This result is derived as a corollary of a sophisticated analysis of the Parisi solution for $F(\beta, h)$ for general mixed $p$-spin Ising models, which gives several interesting consequences, including a general de Almeida--Thouless-style prediction for the phase diagram of all such models with external field. In particular, the authors of \cite{JT2017} showed that for a generalization of $\alpha$ to the mixed $p$-spin case, which we denote here by $\tilde \alpha$, replica symmetry breaking holds for all $(\beta, h)$ such that $\tilde \alpha(\beta, h) > 1$. However, it was recently shown in \cite{MS2026} that this generalized criterion does not characterize the replica-symmetric region for arbitrary mixed \(p\)-spin models, via explicit counterexamples.

Additionally, replica symmetry was shown in the region 
\[
\beta^2\E\big[\sech^2(\beta\sqrt q Z+h)\big]\le 1,
\]
which is strictly contained in the region predicted by de Almeida and Thouless, by a conditional second moment method argument in \cite{BrenneckeYau2022} (extending the previous work \cite{bolthausen2018morita}). We also mention that the analogous problem for the Sherrington--Kirkpatrick model with a centered Gaussian external field was studied in \cite{chen2021almeida}, and a multi-species analogue was considered in \cite{kim2025almeida}.

The main contribution of this paper is the following theorem, which shows that replica symmetry holds for all $(\beta, h)$ such that $\alpha(\beta, h) \le 1$. Together with the results of  \cite{toninelli2002almeida}, this establishes the correctness of the de Almeida--Thouless prediction for all $(\beta, h)$ such that $h >0$.

\begin{theorem}\label{t:main}
Fix $\beta>0$ and $h>0$ such that $\alpha(\beta,h)\le 1$, and let $q\in[0,1]$ solve  \eqref{e:q-fp}. Then $F(\beta,h)=\Phi_{\mathrm{RS}}(\beta,h)$. 
\end{theorem}

The proof of Theorem~\ref{t:main} appears in Section~\ref{s:conclusion}. We give an overview of its main ideas in the next section, after introducing some preliminary concepts. 

Additionally, we include the following proposition concerning the boundary of the replica-symmetric region, which shows that the set of points $(\beta, h)$ such that $\alpha(\beta, h) = 1$ forms a smooth curve. This answers \cite[Question 8.15]{JT2017} for the Sherrington--Kirkpatrick model. The proof is given in the appendix.

\begin{proposition}
\label{prop:predicted-sk-at-smooth-graph}
Define 
\[
\mathcal A_{\mathrm{AT}} =
\big \{
(\beta,h)\in (0,\infty)\times (0,\infty): \alpha(\beta, h) = 1
\big \} \cup \big\{ (1,0) \big\}.
\]
There exists a strictly increasing function $
h_{\mathrm{AT}}:(1,\infty)\to(0,\infty)$  satisfying 
\[
h_{\mathrm{AT}}\in C^\infty\big((1,\infty)\big),
\qquad
\lim_{\beta\rightarrow1^+}h_{\mathrm{AT}}(\beta)=0,
\]
such that
\[
\mathcal A_{\mathrm{AT}}
=
\big\{ \big(\beta,h_{\mathrm{AT}}(\beta)\big):\beta>1 \big\}\cup \{(1,0)\}.
\]
\end{proposition}

\section{Preliminaries}\label{s:preliminaries}

For the proof of Theorem~\ref{t:main}, we fix a pair $(\beta,h)$ satisfying its
hypotheses, abbreviate $\alpha = \alpha(\beta, h)$, and let $q$ be the fixed point given by \eqref{e:q-fp}. 
In this section, we collect some results from previous works required for our arguments.

Write $\Pr([0,1])$ for the set of Borel probability measures on $[0,1]$.  Let
$(W_t)_{0\le t\le 1}$ be a standard Brownian motion on a filtered probability space
$(\Omega,\mathcal F,(\mathcal F_t)_{0\le t\le 1},\Pp)$, where $(\mathcal F_t)$ is the usual
augmentation of its natural filtration.

\subsection{Parisi PDE}

For $\mu\in\Pr([0,1])$, let $u_\mu$ denote the solution of the Parisi PDE
\begin{equation}\label{e:Parisi-PDE}
\partial_t u_\mu(t,x)+\frac{\beta^2}{2}\Bigl(u_{\mu,xx}(t,x)+\mu([0,t])u_{\mu,x}(t,x)^2\Bigr)=0,
\qquad u_\mu(1,x)=\log\cosh x.
\end{equation}
The existence of a solution is guaranteed by the following proposition, which combines the statements of 
\cite[Theorem 4]{JTDP2016} and \cite[Lemma 16]{JTDP2016}; see also \cite[Proposition 2]{ACunique}.

\begin{proposition}\label{p:PDE-wellposed}
For every $\mu\in\Pr([0,1])$, the terminal-value problem \eqref{e:Parisi-PDE} admits a
unique weak solution $u_\mu$.  Moreover, for every $j\ge 1$ one has
$\partial_x^j u_\mu\in C_b([0,1]\times\R)$, for every $j\ge 0$ one has
$\partial_t\partial_x^j u_\mu\in L^\infty([0,1]\times\R)$, and
\[
|u_{\mu,x}(t,x)|\le 1,
\qquad
0<u_{\mu,xx}(t,x)\le 1
\qquad \text{for all }(t,x)\in[0,1]\times\R.
\]
\end{proposition}

The Parisi functional is
\begin{equation}\label{e:Parisi-functional}
\Par(\mu)=\log 2+u_\mu(0,h)-\frac{\beta^2}{2}\int_0^1 s \, \mu([0,s])\,ds.
\end{equation}

The following theorem, due to Talagrand \cite{Talagrand2006}, establishes that the Parisi functional characterizes the limiting free energy of the Sherrington--Kirkpatrick model. Note the result in \cite{Talagrand2006} is not phrased in terms of a PDE; however, the equivalence to the form used here is discussed in \cite[Section 1]{ACunique}.

\begin{theorem}\label{thm:Talagrand}
We have
\[
F(\beta,h)=\inf_{\mu\in\Pr([0,1])}\Par(\mu).
\]
\end{theorem}

\subsection{Variational characterization}
The Parisi PDE is a Hamilton--Jacobi--Bellman equation. Accordingly, its solution $u_\mu(t,x)$ is the value function of a stochastic control problem. For complete details of this equivalence, we refer the reader to \cite{JTDP2016}. For our purposes, we need only the SDE 
\begin{equation}\label{e:AC-SDE}
dX_t^\mu=\beta^2\mu([0,t])u_{\mu,x}(t,X_t^\mu)\,dt+\beta\,dW_t,
\qquad X_0^\mu=h,
\end{equation}
which describes the evolution of the state variable $X_t$ in the stochastic-control representation of the Parisi PDE evaluated under the optimal control. This SDE, and the corresponding variational problem, first appeared in \cite{ACunique}. Because
$u_{\mu,x}$ is bounded and spatially Lipschitz by Proposition~\ref{p:PDE-wellposed}, 
standard Lipschitz SDE theory yields a unique strong solution to \eqref{e:AC-SDE} for every
initial condition; see, for example, \cite[Theorem~5.2.1]{Oksendal2003}.  

For all $\mu\in\Pr([0,1])$, we define 
\begin{equation}\label{e:Gmu-def}
G_\mu(t)=\int_t^1 \frac{\beta^2}{2}\Bigl(\E\big[ u_{\mu,x}(s,X_s^\mu)^2\big]-s\Bigr)\,ds,
\qquad t\in[0,1].
\end{equation}
The following proposition characterizing minimizers of the Parisi functional, taken from \cite[Corollary 3.6]{JT2017}, is a critical technical tool for our work.

\begin{proposition}\label{p:JT}
A probability measure $\mu\in\Pr([0,1])$ minimizes $\Par$ if and only if every point of
$\supp\mu$ is a minimizer of $G_\mu$.
\end{proposition}

\subsection{Proof overview for Theorem~\ref{t:main}}

Proposition~\ref{p:JT} reduces replica symmetry to a concrete statement about the
replica-symmetric candidate $\mu=\delta_q$, where $q$ solves \eqref{e:q-fp}. 
In
particular, it is enough to prove that the associated function $G_{\delta_q}(t)$ is
minimized at $t=q$. 
The proof of Theorem~\ref{t:main} then reduces to a one-dimensional
problem for the sign of $\E[u_x(t,X_t)^2]-t$, where $u$ is the Parisi PDE solution and
$X_t$ is the diffusion from \eqref{e:AC-SDE}. Because $\mu=\delta_q$ has a single atom,
the analysis becomes explicit and naturally splits at the interface $t=q$. 
For $0\le
t<q$, we have $u_x(t,X_t)=\E[\tanh(X_q)\mid\mathcal F_t]$. Since $\E[u_x(q,X_q)^2]=q$,
we can use the conditional Gaussian Poincar\'e inequality and $\alpha(\beta,h)\le 1$ to show
\[
\E[u_x(t,X_t)^2]-t\ge(1-\alpha)(q-t)\ge 0
\qquad \text{for all }0\le t\le q,
\]
which is exactly the sign needed in the variational argument.

For $q\le t\le 1$, the explicit identity $u_x(t,x)=\tanh x$ motivates introducing
$m_t=\tanh(X_t)$. 
Then the desired inequality becomes $\E[m_t^2]\le t$, and It\^o's
formula shows that the drift cancels exactly, so $m_t$ solves the drift-free diffusion
\[
dm_t=\beta(1-m_t^2)\,dW_t.
\]
When $\beta^2(1-q)\le 1$, the map $t\mapsto\E[m_t^2]$ is nondecreasing and $\E[m_q^2]=q$, so
\[
\frac{d}{dt}\E[m_t^2]
=\beta^2\E[(1-m_t^2)^2]
\le\beta^2\big(1-\E[m_t^2]\big)
\le\beta^2(1-q)\le 1.
\]
Integrating from $q$ to $t$ proves $\E[m_t^2]\le t$ in this case.

The remaining regime, where $\beta^2(1-q)>1$, is more delicate. 
We begin by proving the
structural inequality $h<\beta^2q$. 
We then derive an explicit formula for the evolution
beyond time $q$ by expressing the relevant semigroup through Brownian motion and a
$\cosh$ reweighting. 
Then a short calculation using a symmetric
double-integral identity gives
\[
\E[(1-m_t^2)^2]
\le e^{-\beta^2(t-q)/2}\E[(1-m_q^2)^2].
\]
From this we infer that
\[
\frac{d}{dt}\E[m_t^2]
\le\alpha e^{-\beta^2(t-q)/2}\le 1,
\]
which yields $\E[m_t^2]\le t$ for all $t\ge q$. Consequently, $G_{\delta_q}$ is
nonincreasing on $[0,q]$ and nondecreasing on $[q,1]$, so its minimum is attained at $q$.

\section{Explicit formulas for the RS ansatz $\mu=\delta_q$}

We recall that $(\beta,h)$ was fixed in the previous section, and we abbreviate 
\[
u=u_{\delta_q},
\qquad
X=X^{\delta_q}.
\]

\subsection{Explicit PDE and SDE}
We begin by writing down the explicit PDE and SDE associated with $u$. 

\begin{proposition}\label{prop:explicit-RS}
For $\mu=\delta_q$, we have
\begin{equation}\label{e:ux-hard}
u(t,x)=\frac{\beta^2}{2}(1-t)+\log\cosh x,
\qquad
u_x(t,x)=\tanh x,
\qquad
u_{xx}(t,x)=\sech^2 x,
\qquad q\le t\le 1,
\end{equation}
and
\begin{equation}\label{e:u-soft}
u(t,x)=\frac{\beta^2}{2}(1-q)+\E\big[ \log\cosh(x+\beta\sqrt{q-t}\,Z )\big],
\qquad 0\le t\le q.
\end{equation}
Moreover,
\begin{equation}\label{e:sde-soft-hard}
dX_t=
\begin{cases}
\beta\,dW_t, & 0\le t<q,\\
\beta^2\tanh(X_t)\,dt+\beta\,dW_t, & q\le t\le 1,
\end{cases}
\qquad X_0=h,
\end{equation}
and 
\begin{equation}\label{e:Xq-law}
X_q\stackrel d= h+\beta\sqrt q Z.
\end{equation}
\end{proposition}

\begin{proof}
Since $\delta_q([0,t])=1$ for $t\ge q$, the PDE \eqref{e:Parisi-PDE} becomes
\[
\partial_t u+\frac{\beta^2}{2}(u_{xx}+u_x^2)=0,
\qquad u(1,x)=\log\cosh x,
\]
on $[q,1]$.  The function in \eqref{e:ux-hard} satisfies the terminal condition and the PDE because
$u_{xx}+u_x^2=\sech^2x+\tanh^2x=1$.

At $t=q$ this gives
\[
u(q,x)=\frac{\beta^2}{2}(1-q)+\log\cosh x.
\]
Since $\delta_q([0,t])=0$ for $t<q$, the PDE reduces on $[0,q)$ to the backward heat equation
\[
\partial_t u+\frac{\beta^2}{2}u_{xx}=0,
\qquad u(q,x)=\frac{\beta^2}{2}(1-q)+\log\cosh x,
\]
and the heat semigroup yields \eqref{e:u-soft}.

For the SDE, \eqref{e:AC-SDE} becomes
\[
dX_t=\beta^2\delta_q([0,t])u_x(t,X_t)\,dt+\beta\,dW_t,
\qquad X_0=h,
\]
which is exactly \eqref{e:sde-soft-hard}, and \eqref{e:Xq-law} follows from the driftless segment
$X_t=h+\beta W_t$ on $[0,q]$.
\end{proof}

The explicit formulas in Proposition~\ref{prop:explicit-RS} show that $u$ is $C^{1,2}$ on $(0,q)\times\mathbb R$ and on $(q,1)\times\mathbb R$. Moreover, the spatial derivatives extend continuously across the interface $t=q$: for every $x\in\mathbb R$,
\begin{equation}\label{e:agree}
u_x(q-,x)=u_x(q+,x)=\tanh x,
\qquad
u_{xx}(q-,x)=u_{xx}(q+,x)=\sech^2 x.
\end{equation}

\section{The interval $0\le t\le q$}

\begin{proposition}\label{prop:left-sign}
For every $0\le t\le q$,
\begin{equation}\label{e:left-sign}
\E\big[ u_x(t,X_t)^2\big]-t\ge(1-\alpha)(q-t)\ge 0.
\end{equation}
\end{proposition}

\begin{proof}
At $t=q$, \eqref{e:ux-hard} gives $u_x(q,x)=\tanh x$, hence
\[
\E\big[u_x(q,X_q)^2\big]=\E\big[\tanh^2(X_q)\big]=q
\]
by the fixed-point equation \eqref{e:q-fp} and \eqref{e:Xq-law}.
Fix $0\le t<q$. On $[0,q]$ we have $X_t=h+\beta W_t$ by
\eqref{e:sde-soft-hard}, hence
\[
X_q=X_t+\beta(W_q-W_t).
\]
The increment $W_q-W_t$ is independent of $\mathcal F_t$ and Gaussian
with variance $q-t$, so the conditional law of $X_q$ given $\mathcal F_t$ is
$X_t+\beta\sqrt{q-t}\,Z$. Differentiating \eqref{e:u-soft} gives
\begin{equation}\label{e:ux-soft-cond}
u_x(t,X_t)=\E[\tanh(X_q)\mid\mathcal F_t].
\end{equation}
By the conditional Gaussian Poincar\'e inequality,
\[
\operatorname{Var}\big(\tanh(X_q)\mid\mathcal F_t\big)
\le\beta^2(q-t)\E\big[\sech^4(X_q)\mid\mathcal F_t\big].
\]
Taking expectations and using \eqref{e:ux-soft-cond}, we obtain
\[
\begin{aligned}
q-\E\big[u_x(t,X_t)^2\big]
&=\E\big[\operatorname{Var}(\tanh(X_q)\mid\mathcal F_t)\big]\\
&\le\beta^2(q-t)\E\big[\sech^4(X_q)\big]\\
&=\alpha(q-t).
\end{aligned}
\]
Rearranging yields \eqref{e:left-sign}.
\end{proof}

\section{The interval  $q\le t\le 1$}

Set
\[
m_t=\tanh(X_t), \qquad q\le t\le 1.
\]

\subsection{The diffusion for $m_t$}

We begin by deriving basic properties of $m_t$. 

\begin{proposition}\label{prop:m-diffusion}
On $[q,1]$, 
\begin{equation}\label{e:m-sde}
dm_t=\beta(1-m_t^2)\,dW_t.
\end{equation}
Further,
\begin{equation}\label{e:m-second-moment}
\frac{d}{dt}\E[m_t^2]=\beta^2\E\big[(1-m_t^2)^2\big],
\qquad
\E[m_q^2]=q,
\qquad q\le t\le1,
\end{equation}
where the derivatives at the endpoints are understood as one-sided derivatives.
\end{proposition}

\begin{proof}
We apply It\^o's formula to $m_t=\tanh(X_t)$ using the SDE in \eqref{e:sde-soft-hard}. 
Since $(\tanh)'=\sech^2$ and $(\tanh)''=-2\tanh\,\sech^2$, the drift cancels and one obtains
\eqref{e:m-sde}.  Applying It\^o's formula to $m_t^2$ then gives
\[
d(m_t^2)=2\beta m_t(1-m_t^2)\,dW_t+\beta^2(1-m_t^2)^2\,dt.
\]
Taking expectations,
\[
\E[m_t^2]=\E[m_q^2]+\beta^2\int_q^t \E\big[ (1-m_s^2)^2\big] \,ds.
\]
Since the integrand is bounded and continuous in $s$, the map $t\mapsto \E[m_t^2]$ is $C^1$ and we have the first equality in \eqref{e:m-second-moment}.  At $t=q$, $m_q=\tanh(X_q)$ and \eqref{e:Xq-law} give
\[
\E[m_q^2]=\E\big[\tanh^2(\beta\sqrt q Z+h)\big]=q,
\]
as desired.
\end{proof}

\subsection{Case I:  $\beta^2(1-q)\le 1$}

\begin{proposition}\label{prop:hard-small}
If 
$\beta^2(1-q)\le 1$, 
then  $\E[m_t^2]\le t$ for all $t\in[q,1]$.
\end{proposition}

\begin{proof}
By Proposition~\ref{prop:m-diffusion}, the map $t\mapsto \E[m_t^2]$ is
nondecreasing, so $\E[m_t^2]\ge\E[m_q^2]=q$ for $t\in[q,1]$.
Since $m_t^2\in[0,1]$, one has $(1-m_t^2)^2\le1-m_t^2$, and therefore
\[
\begin{aligned}
\frac{d}{dt}\E[m_t^2]
&=\beta^2\E\big[(1-m_t^2)^2\big]\\
&\le\beta^2\big(1-\E[m_t^2]\big)\\
&\le\beta^2(1-q)\le 1.
\end{aligned}
\]
Integrating from $q$ to $t$ gives
\[
\E[m_t^2]\le q+(t-q)=t,
\]
as desired.
\end{proof}

\subsection{Case II: $\beta^2(1-q)>1$}

Set $
\sigma^2=\beta^2 q$. 
Because $h>0$, the fixed-point equation \eqref{e:q-fp} rules out $q=0$; hence $\sigma>0$.

\begin{proposition}\label{prop:h-less-beta2q}
If $\beta^2(1-q)>1$, then
$h<\sigma^2$.
\end{proposition}

\begin{proof}
Assume toward a contradiction that $h\ge \sigma^2$. 
For $a\ge0$, write $p_a$ for the density of $N(a,\sigma^2)$. For $a>0$,
\[
\frac{d}{da}\E\big[\tanh^2(a+\sigma Z)\big]
=2\int_0^\infty \tanh y\,\sech^2 y
\bigl(p_a(y)-p_a(-y)\bigr)\,dy>0,
\]
because $p_a(y)>p_a(-y)$ for $a,y>0$. Therefore, setting $Y=\sigma^2+\sigma Z$,
\[
\E\big[\tanh^2Y\big]\le\E\big[\tanh^2(h+\sigma Z)\big]=q.
\]
Now
\[
(1-\tanh y)p_{\sigma^2}(y)
=\frac{e^{-\sigma^2/2}}{\sigma\sqrt{2\pi}}
\frac{e^{-y^2/(2\sigma^2)}}{\cosh y}
\]
is even. Hence
\[
\E\big[Y(1-\tanh Y)\big]=0,
\qquad
\E[Y\tanh Y]=\E[Y]=\sigma^2.
\]
By Cauchy--Schwarz,
\[
\sigma^4=\bigl(\E[Y\tanh Y]\bigr)^2
\le\E[Y^2]\E\big[\tanh^2Y\big]
\le\sigma^2(1+\sigma^2)q.
\]
Since $\sigma^2=\beta^2q$ and $q>0$, this yields $\beta^2\le1+\beta^2q$,
that is, $\beta^2(1-q)\le1$, contradicting the assumption that $\beta^2(1-q) >1$. Hence
the conclusion must hold.
\end{proof}

Motivated by Proposition~\ref{prop:h-less-beta2q}, we now study the interval $q \le t \le 1$ under the condition $h\le \sigma^2$.

Let $P_\lambda$ denote the heat semigroup,
\[
(P_\lambda\phi)(x)=\E\big[\phi(x+\sqrt\lambda\,G)\big],
\qquad G\sim N(0,1).
\]
Using $P_\lambda$, we can describe the distribution of $X_t$ conditional on its value at $t=q$. By a direct calculation, we implement a Doob transform to write the diffusion in \eqref{e:sde-soft-hard}, which has a nonlinear drift, as a Gaussian expectation with an explicit weight function.

\begin{proposition}\label{prop:Doob}
Let $t=q+\lambda/\beta^2$ with $0\le \lambda\le \beta^2(1-q)$, and fix $x\in\R$.
For every bounded measurable function $\psi:\R\to\R$,
\begin{equation}\label{e:Doob}
\E\bigl[\psi(X_t)\mid X_q=x\bigr]
=\frac{e^{-\lambda/2}}{\cosh x}
P_\lambda\big(\cosh(\cdot)\psi(\cdot)\big)(x).
\end{equation}
\end{proposition}

\begin{proof}
Set
\[
Y_\lambda=X_{q+\lambda/\beta^2},
\qquad
B_\lambda=\beta\bigl(W_{q+\lambda/\beta^2}-W_q\bigr).
\]
Then $B$ is a standard Brownian motion and, by \eqref{e:sde-soft-hard},
\[
dY_\lambda=\tanh(Y_\lambda)\,d\lambda+dB_\lambda.
\]
It follows that the generator of $Y$ is
\[
L\phi=\frac12\phi''+(\tanh x)\phi'.
\]
A direct calculation gives
\[
(\cosh x\,\phi)''
=\cosh x\,\phi+2\sinh x\,\phi'+\cosh x\,\phi'',
\]
and hence
\[
L\phi
=\frac{1}{2\cosh x}\big( (\cosh x)\phi\big)''-\frac12\phi.
\]

Define
\[
(T_\lambda\phi)(x)
=\frac{e^{-\lambda/2}}{\cosh x}
P_\lambda\big(\cosh(\cdot)\phi(\cdot)\big)(x).
\]
We see that $T_\lambda$ is positivity preserving, and
\[
\begin{aligned}
P_\lambda(\cosh)(x)
&=\E\big[\cosh(x+\sqrt\lambda\,G)\big] =\cosh x\,\E[\cosh(\sqrt\lambda\,G)]
=e^{\lambda/2}\cosh x,
\end{aligned}
\]
so that $T_\lambda1=1$.

For every $\phi\in C_c^\infty(\R)$, the function $u(r,y)=(T_r\phi)(y)$ is smooth on $[0,\infty)\times\R$ and bounded by $\|\phi\|_\infty$. Since $w=\cosh(\cdot)\,u$ satisfies $\partial_r w=\frac12\partial_y^2w-\frac12w$, the identity for $L$ above gives $\partial_r u=Lu$ with $u(0,\cdot)=\phi$. By uniqueness for Kolmogorov's backward equation \cite[Theorem~8.1.1]{Oksendal2003}, $u(\lambda,x)=\E[\phi(Y^x_\lambda)]$, where $Y^x$ solves the SDE above with $Y^x_0=x$. Since $C_c^\infty(\R)$ test functions suffice to uniquely determine probability measures, $\psi\mapsto(T_\lambda\psi)(x)$ is the law of $Y^x_\lambda$. Finally, since $Y_0=X_q$ and $B$ is independent of $\mathcal F_q$, \eqref{e:Doob} follows from the Markov property.
\end{proof}

\begin{proposition}\label{prop:hard-hsmall}
Assume $h\le \sigma^2$ and $\alpha \le 1$. 
Then for all $t\in[q,1]$, we have $
\E[m_t^2]\le t$. 
The inequality is strict for $t>q$.
\end{proposition}

\begin{proof}
Let $p$ be the density of $N(h,\sigma^2)$ and write $k(x)=\sech^3 x$. 
Define
\[
a_2(\lambda)=\E\big[\sech^4(X_{q+\lambda/\beta^2})\big],
\qquad 0\le \lambda\le \beta^2(1-q).
\] 
Applying \eqref{e:Doob} with $\psi(y)=\sech^4 y$ gives
\[
a_2(\lambda)=e^{-\lambda/2}\int_{\R}\frac{p(x)}{\cosh x}
(P_\lambda k)(x)\,dx.
\]
Since $P_\lambda k$ is even,
\[
e^{\lambda/2}a_2(\lambda)=\int_{\R}w(x)(P_\lambda k)(x)\,dx,
\]
where
\[
w(x)=\frac{p(x)+p(-x)}{2\cosh x}
=\frac{1}{\sigma\sqrt{2\pi}}\exp\left( -\frac{x^2+h^2}{2\sigma^2}\right)
\frac{\cosh(hx/\sigma^2)}{\cosh x}.
\]
Set $c=h/\sigma^2\le 1$. Then for $x\ge 0$,
\[
c\tanh(cx)\le c\tanh x\le \tanh x,
\]
so, for all $x>0$, 
\[
\frac{w'(x)}{w(x)}
=-\frac{x}{\sigma^2}+c\tanh(cx)-\tanh x
\le-\frac{x}{\sigma^2}<0.
\]
Then $w$ is even and decreasing on $[0,\infty)$.

For $\lambda>0$, write $\varphi_\lambda$ for the centered Gaussian density
with variance $\lambda$. By the symmetry of $\varphi_\lambda$ and Fubini's theorem,
\[
\begin{aligned}
a_2(0)-e^{\lambda/2}a_2(\lambda)=\frac12\iint_{\R^2}\varphi_\lambda(x-y)
\bigl(w(x)-w(y)\bigr)\bigl(k(x)-k(y)\bigr)\,dx\,dy
\ge0.
\end{aligned}
\]
The last inequality follows because $w$ and $k$ are even and decreasing on
$[0,\infty)$, so the two differences have the same sign for every $x,y\in\R$.
The integrals are absolutely convergent since $w$ is integrable and $k$ is
bounded. Since $a_2(0)=\alpha/\beta^2$, this proves \begin{equation}\label{e:a2-dec}
a_2(\lambda)\le e^{-\lambda/2}a_2(0)
=\frac{\alpha}{\beta^2}e^{-\lambda/2}
\end{equation}
for
$\lambda>0$. At $\lambda=0$ the inequality is an equality.

By Proposition~\ref{prop:m-diffusion} and \eqref{e:a2-dec}, and the identity $(1-m_t^2)^2 = \sech^4(X_t)$, 
\[
\frac{d}{dt}\E[m_t^2]
=\beta^2a_2\big(\beta^2(t-q)\big)
\le\alpha e^{-\beta^2(t-q)/2}.
\]
Integrating from $q$ to $t$ and using $\E[m_q^2]=q$ gives the stated bound.
Since $\alpha\le 1$ and $1-e^{-z}\le z$ for $z\ge0$, it follows that
$\E[m_t^2]\le t$, with strict inequality for $t>q$.
\end{proof}

\begin{proposition}\label{prop:hard-full}
For all $ t\in[q,1]$, we have
\begin{equation}\label{e:hard-full}
\E\big[ u_x(t,X_t)^2\big] \le t.
\end{equation}
\end{proposition}

\begin{proof}
If $\beta^2(1-q)\le 1$, the claim follows from Proposition~\ref{prop:hard-small}.  If instead
$\beta^2(1-q)>1$, then Proposition~\ref{prop:h-less-beta2q} gives $h<\beta^2 q$, so
Proposition~\ref{prop:hard-hsmall} yields the claim.  We also used that on $[q,1]$ one has
$u_x(t,x)=\tanh x$ by \eqref{e:ux-hard}.
\end{proof}

\section{Conclusion}\label{s:conclusion}

\begin{proposition}\label{prop:G-min}
For $\mu=\delta_q$, the function
\[
G_{\delta_q}(t)=\int_t^1 \frac{\beta^2}{2}\bigl(\E\big[u_x(s,X_s)^2\big]-s\bigr)\,ds
\]
attains its global minimum at $t=q$.
\end{proposition}

\begin{proof}
By definition, for every $t\in[0,1]$,
\[
G_{\delta_q}(t)-G_{\delta_q}(q)
=\frac{\beta^2}{2}\int_t^q
\bigl(\E\big[u_x(s,X_s)^2\big]-s\bigr)\,ds.
\]
For $t\le q$, the integral is nonnegative by Proposition~\ref{prop:left-sign}.
For $t\ge q$, reversing the limits and using Proposition~\ref{prop:hard-full}
gives the same conclusion. Therefore $G_{\delta_q}$ is globally minimized at $q$.
\end{proof}

\begin{proof}[Proof of Theorem~\ref{t:main}]
By Proposition~\ref{prop:G-min}, the function $G_{\delta_q}$ attains its minimum at the unique point
of $\supp\delta_q$, namely at $q$.  Proposition~\ref{p:JT} therefore implies that $\delta_q$ minimizes $\Par$.

We now apply the Parisi formula, Theorem~\ref{thm:Talagrand}, to obtain
\[
F(\beta,h)=\inf_{\mu\in\Pr([0,1])}\Par(\mu)=\Par(\delta_q).
\]
Using \eqref{e:u-soft} at $t=0$, we have
\[
u(0,h)=\frac{\beta^2}{2}(1-q)+\E\big[\log\cosh(h+\beta\sqrt q Z)\big].
\]
Also,
\[
\frac{\beta^2}{2}\int_0^1 s\,\delta_q([0,s])\,ds
=\frac{\beta^2}{2}\int_q^1 s\,ds
=\frac{\beta^2}{4}(1-q^2).
\]
Thus
\begin{align*}
\Par(\delta_q)
&=\log 2+\frac{\beta^2}{2}(1-q)+\E\big[\log\cosh(h+\beta\sqrt q Z)\big]-\frac{\beta^2}{4}(1-q^2)\\
&=\Phi_{\mathrm{RS}}(\beta,h).
\end{align*}
Hence $F(\beta,h)=\Phi_{\mathrm{RS}}(\beta,h)$.
\end{proof}

\begin{remark}
The uniqueness theorem for the Parisi minimizer \cite[Corollary 1]{ACunique}, together with 
Theorem~\ref{t:main}, implies the stronger statement that the Parisi minimizer is exactly
$\delta_q$.
\end{remark}

\begin{appendix}
\section*{Appendix}\label{appn}
\begin{proof}[Proof of  Proposition~\ref{prop:predicted-sk-at-smooth-graph}]
For all \(h\ge0\) and \(t\ge0\), define
\[
A(h,t) =\mathbb E\big[ \operatorname{sech}^2(h+\sqrt t\,Z) \big],
\qquad
C(h,t) =\mathbb E\big[ \operatorname{sech}^4(h+\sqrt t\,Z) \big],
\]
and
\[
F(h,t) =tC(h,t)+A(h,t)-1.
\]

We first analyze the case $h>0$. Suppose $\alpha(\beta, h) = 1$. Then $q >0$ by \eqref{e:q-fp}. 
Set $t = \beta^2q$. 
Then
\[
q=1-A(h,t),
\qquad
1=\beta^2C(h,t).
\]
Therefore
\[
tC(h,t)
=
\beta^2qC(h,t)
=
q,
\]
and hence
\[
F(h,t)=q+A(h,t)-1=0.
\]

Conversely, suppose \(h>0\), \(t>0\), and \(F(h,t)=0\). Define
\[
\beta^2=\frac1{C(h,t)},
\qquad
q =tC(h,t).
\]
Since \(F(h,t)=0\), we have
\[
q=tC(h,t)=1-A(h,t)
=\mathbb E\big[ \tanh^2(h+\sqrt t\,Z) \big],
\]
so \(q\in(0,1)\). Also, 
\[
\beta\sqrt q
=
\frac1{\sqrt{C(h,t)}}\sqrt{tC(h,t)}
=
\sqrt t.
\]
Then, using $q =tC(h,t)$, 
\[
\beta^2\mathbb E \big[ \operatorname{sech}^4(h+\beta\sqrt q\,Z)\big]=1.
\]
Hence the positive-field part of $\mathcal A_{\mathrm{AT}}$ is exactly the image of the zero set
\[
F(h,t)=0,\qquad h>0,\ t>0,
\]
under the map $
(h,t)\mapsto (C(h,t)^{-1/2},h)$.

We now solve \(F(h,t)=0\) for \(h\) as a function of \(t\) after establishing some intermediate facts. Fix \(h\ge0\) and \(t>0\), and write $Y=h+\sqrt t\,Z$.
Abbreviate
\[
A=\mathbb E[\sech^2Y],\qquad
C=\mathbb E[\sech^4Y],
\]
\[
u=\mathbb E[\tanh(Y)\sech^2(Y)],
\qquad
v=\mathbb E[\tanh(Y)\sech^4(Y)],
\]
and
\[
B=A-C,\qquad T=\mathbb E[Y\tanh(Y)\sech^4(Y)].
\]
Also set
\[
s(y)=\sech^2y,
\qquad
m(y)=\tanh y.
\]
For all \(h>0\), we have \(u,v>0\). Indeed, if \(g\) is odd and \(g(y)>0\) for \(y>0\), then, writing \(p_{h,t}\) for the density of  the Gaussian \(N(h,t)\),
\[
\mathbb E\big[ g(Y)\big]
=
\int_0^\infty g(y)\bigl(p_{h,t}(y)-p_{h,t}(-y)\bigr)\,dy>0.
\]
Applying this fact with \(g=ms\) and \(g=ms^2\) gives the positivity claim.
At \(h=0\), one has \(u=v=0\) by symmetry. Also, \(B,T>0\), since
\(Y\) is nondegenerate and \(y\tanh y>0\) for \(y\ne0\).
Differentiating in \(h\), we get
\[
A_h=-2u,\qquad C_h=-4v,\qquad F_h=-2u-4tv.
\]
Thus \(h\mapsto F(h,t)\) is strictly decreasing on \((0,\infty)\).

Gaussian integration by parts gives
\[
\mathbb E\big[(ms^2)'(Y)\big]
=\frac{1}{t}\mathbb E\big[(Y-h)m(Y)s(Y)^2\big]
=\frac{T-hv}{t}.
\]
Since \((s^2)'=-4ms^2\), partial differentiation in \(t\)  gives
\[
C_t=\frac12\mathbb E\big[(s^2)''(Y)\big]=\frac{2}{t}(hv-T).
\]
Differentiating the other quantities directly, we have
\[
A_t=2A-3C,
\qquad
F_t=C+tC_t+A_t=2(B-T+hv).
\]
We have
\[
B-T=\mathbb E\big[s(Y)m(Y)\big(m(Y)-Ys(Y)\big)\big].
\]
For \(y>0\), we have $
m(y)-ys(y)>0$; this expression vanishes at \(y=0\), and its derivative is
\[
s(y)-s(y)-ys'(y)=-ys'(y)=2ys(y)m(y)>0.
\]
Since $m(y)$ is odd, we also have \(m(y)(m(y)-ys(y))>0\) for \(y<0\).
Since \(Y\) is nondegenerate, \(B-T>0\). Also, \(hv\ge0\), so
\[
F_t(h,t)>0 \qquad (h\ge0,\ t>0).
\]
Since \(F(0,0)=0\) and \(F\) is continuous at \(t=0\), it follows that
\(F(0,t)>0\) for every \(t>0\).
For each fixed \(t>0\), dominated convergence gives $A(h,t),C(h,t)\to0$
as $h\rightarrow\infty$, and hence
\[
\lim_{h\rightarrow\infty}F(h,t)=-1.
\]
We have shown that \(F(0,t)>0\) and \(h\mapsto F(h,t)\) is strictly
decreasing. Hence there exists a unique $h(t)\in(0,\infty)$ such that
$F(h(t),t)=0$.
Since \(F_h(h(t),t)<0\), the implicit function theorem shows that the map
$t\mapsto h(t)$ is $C^\infty((0,\infty))$.

For the remainder of the proof, set \(h=h(t)\) in the notation above.
Differentiating \(F(h(t),t)=0\) gives
\[
h'(t)=-\frac{F_t}{F_h}
=\frac{B-T+hv}{u+2tv}>0
\qquad (t>0).
\]
So \(h(t)\) is strictly increasing on \((0,\infty)\).

Define
\[
\beta(t)=C(h(t),t)^{-1/2}.
\]
We now prove that \(\beta(t)\) is strictly increasing. It suffices to show that
\[
\frac d{dt}C\big(h(t),t\big)<0.
\]
Since \((ms)'=3s^2-2s\), Gaussian integration by parts gives
\[
\mathbb E[Y m(Y)s(Y)]
=hu+t\mathbb E[(ms)'(Y)]
=hu+t(3C-2A).
\]
Along the zero set, \(tC=1-A=\mathbb E[m(Y)^2]\). Therefore
\[
\begin{aligned}
2tB-hu
&=tC-\mathbb E[Y m(Y)s(Y)]\\
&=\mathbb E\big[m(Y)\big(m(Y)-Ys(Y)\big)\big]>0,
\end{aligned}
\]
where the strict inequality follows from the pointwise inequality used above.
Differentiating \(C(h(t),t)\) along the curve and using the formula for
\(h'(t)\), we obtain
\[
\begin{aligned}
\frac d{dt}C(h(t),t)
&=C_t+C_hh'(t)\\
&=\frac{2(hv-T)}{t}-\frac{4v(B-T+hv)}{u+2tv}\\
&=-\frac{2\big[uT+v(2tB-hu)\big]}{t(u+2tv)}<0.
\end{aligned}
\]
The strict inequality follows from \(u,v,T>0\) and \(2tB-hu>0\). Hence
\[
\beta'(t)
=-\frac{1}{2C(h(t),t)^{3/2}}\frac d{dt}C(h(t),t)>0,
\]
and we conclude that \(\beta(t)\) is strictly increasing on \((0,\infty)\).

We now identify the range of \(\beta(t)\). Since \(h(t)\) is positive and
strictly increasing, it has a finite limit \(h_0\ge0\) as \(t\rightarrow0\).
By dominated convergence in \(F(h(t),t)=0\),
\[
0=F(h_0,0)=\sech^2h_0-1,
\]
so \(h_0=0\). Again by dominated convergence, $C(h(t),t)\to1$, and hence
$\beta(t)\to1$ as $t\rightarrow0$. 
On the other hand, along the zero set of $F$,
\[
tC(h(t),t)=1-A(h(t),t)<1.
\]
Thus
\[
\beta(t)^2=\frac1{C(h(t),t)}>t,
\]
so $
\beta(t)\to\infty$ as $t\rightarrow \infty$. 
Since \(\beta(t)\) is continuous and strictly increasing, it maps
\((0,\infty)\) bijectively onto \((1,\infty)\).

Let \(t=t(\beta)\) be the inverse map, and define
\[
h_{\mathrm{AT}}(\beta)=h\big(t(\beta)\big).
\]
Since both \(h(t)\) and \(t(\beta)\) are strictly increasing,
\(h_{\mathrm{AT}}\) is strictly increasing on \((1,\infty)\).
By the inverse function theorem, $h_{\mathrm{AT}}\in C^\infty((1,\infty))$.
The positive-field part of \(\mathcal A_{\mathrm{AT}}\) is therefore exactly
\[
\big \{(\beta,h_{\mathrm{AT}}(\beta)):\beta>1 \big\}.
\]
Moreover, since \(h(t)\to0\) and \(\beta(t)\to1\) as \(t\rightarrow0\), we have
\[
\lim_{\beta\rightarrow 1^+}h_{\mathrm{AT}}(\beta)=0.
\]

By definition, the only point with $h=0$ included in
\(\mathcal A_{\mathrm{AT}}\) is \((1,0)\). Combining this with the
positive-field graph representation proves the claim.
\end{proof}
\end{appendix}

\subsection*{Acknowledgment}
This paper was written with the assistance of large language models. This assistance included suggesting arguments, helping with drafting and revising the manuscript, and coding simulations to check the results.

\subsection*{Funding}
The author was partially supported by NSF grant DMS-2450004.

\printbibliography

\end{document}